\documentclass[12pt,draft, reqno]{amsart}
\usepackage{amsmath}
\usepackage{amssymb}
\usepackage{amsthm}
\usepackage{bm}
\numberwithin{equation}{section}
\newtheorem{theorem}{Theorem}[section]

\newtheorem{proposition}[theorem]{Proposition}

\newtheorem{definition}[theorem]{Definition}

\newcommand{\al}{\alpha}

\newcommand{\gm}{\gamma}

\newcommand{\ep}{\varepsilon}

\newcommand{\tht}{\theta}

\newcommand{\ld}{\lambda}

\newcommand{\q}{\quad}
\newcommand{\qq}{\qquad}

\newcommand{\wh}{\widehat}

\newcommand{\R}{\mathbb{R}}
\newcommand{\N}{\mathbb{N}}

\newcommand{\rd}{\mathbb R ^d}

\newcommand{\va}{{\Vec{a}}}

\newcommand{\vc}{{\Vec{c}}}
\newcommand{\vs}{{\Vec{s}}}
\newcommand{\vt}{{\Vec{t}}}

\newcommand{\vsig}{{\Vec{\sigma}}}

\allowdisplaybreaks
\begin{document}
\title[Multivariable finite Euler products]{Behaviors of multivariable finite Euler products in probabilistic view}
\author[T.~Aoyama and T.~Nakamura]{Takahiro Aoyama \& Takashi Nakamura}
\address{Department of Mathematics Faculty of Science and Technology \\Tokyo University of Science Noda, CHIBA 278-8510 JAPAN}
\email{aoyama$\_$takahiro$@$ma.noda.tus.ac.jp, nakamura$\_$takashi$@$ma.noda.tus.ac.jp}
\subjclass[2010]{Primary 60E07, Secondary 11M41}
\keywords{characteristic function, finite Euler product, infinite divisibility}
\maketitle

\begin{abstract}
Finite Euler product is known to be one of the classical zeta functions in number theory.
In \cite{AN12e}, \cite{AN12s} and \cite{AN11k}, we have introduced some multivariable zeta functions and studied their definable probability distributions 
on $\rd$.
They include functions which generate infinitely divisible, not infinitely divisible characteristic functions and not even to be characteristic functions.
In this paper, we treat some multivariable finite Euler products and show how they behave in view of such properties.
\end{abstract}

\section{Introduction}

Infinitely divisible distributions are known as one of the most important class of distributions in probability theory.
They are the marginal distributions of stochastic processes having independent and stationary increments such as Brownian motion and Poisson processes. 
In 1930's, such stochastic processes were well-studied by P. L\'evy and now we usually call them L\'evy processes.
We can find the detail of L\'evy processes in \cite{S99}.

\vskip3mm

In this section, we mention some known properties of infinitely divisible distributions.

\begin{definition}[Infinitely divisible distribution]
A probability measure $\mu$ on $\rd$ is infinitely divisible if,
for any positive integer $n$, there is a probability measure $\mu_n$ on 
$\rd$
such that
\begin{equation*}
\mu=\mu_n^{n*},
\end{equation*}
where $\mu_n^{n*}$ is the $n$-fold convolution of $\mu_n$.
\end{definition}

It should be noted that Normal, degenerate, Poisson and compound Poisson distributions are infinitely divisible.

\vskip3mm

Denote by $ID(\rd)$ the class of all infinitely divisible distributions on $\rd$.
Let $\wh{\mu}(t):=\int_{\rd}e^{{\rm i}\langle t,x\rangle}\mu (dx),\, t\in\rd,$ be the characteristic function of a distribution $\mu$, 
where $\langle\cdot ,\cdot\rangle$ is the inner product.

\vskip3mm

The following is well-known.

\begin{proposition}[L\'evy--Khintchine representation (see, e.g.\,\cite{S99})]\label{pro:LK1}
$(i)$ If $\mu\in ID(\rd)$, then
\begin{equation}\label{INF}
\wh{\mu}(t) = \exp\left[-\frac{1}{2} \langle t,At \rangle+{\rm i}
\langle \gm,t \rangle+\int_{\rd}\left(e^{{\rm i} \langle t,x \rangle }-1-
\frac{{\rm i} \langle t,x \rangle}{1+|x|^2}\right)\nu(dx)\right],\,t\in\rd,
\end{equation}
where $A$ is a symmetric nonnegative-definite $d \times d$ matrix,
$\nu$ is a measure on $\rd$ satisfying
\begin{equation}\label{lev}
\nu(\{0\}) = 0 \,\,and\,\,\int_{\rd} (|x|^{2} \wedge 1) \nu(dx) < \infty,
\end{equation}
and $\gm\in\rd$.

$(ii)$ The representation of $\wh{\mu}$ in $(i)$ by $A, \nu,$ and $\gm$ is
unique.

$(iii)$ Conversely, if $A$ is a symmetric nonnegative-definite $d\times d$
matrix, $\nu$ is a measure satisfying $\eqref{lev}$, and $\gm\in\rd$, then
there exists an infinitely divisible distribution $\mu$ whose characteristic function is given by $\eqref{INF}$.
\end{proposition}

The measure $\nu$ and $(A, \nu ,\gm)$ are called the L\'evy measure and the L\'evy--Khintchine triplet of $\mu\in I(\rd)$, respectively.

\vskip3mm

The study of infinite divisibility has been a major subject in probability theory and also gives us many applications in the study of stochastic differential equations as well.
Stationary distributions of a generalized Ornstein--Uhlenbeck process associated with a certain bivariate L\'evy process is studied 
by Lindner and Sato \cite{LSa}.
Afterwards, they extended them by defining a sequence of bivariate L\'evy processes given above in \cite{LSm}.
(Detail of their definitions, see \cite{LSa} and \cite{LSm}.)
What was interesting in their results in \cite{LSm} is that there appear non-infinitely divisible distributions including a certain known class of distributions by its expansion.
That class is called {\it quasi-infinitely divisible distributions} which is defined as follows.

\begin{definition}[Quasi-infinitely divisible distribution]
A distribution $\mu$ on $\rd$ is called {\it quasi-infinitely divisible} if it has a form of \eqref{INF} and the corresponding measure $\nu$ is a signed 
measure on $\rd$ with total variation measure $|\nu|$ satisfying $\nu(\{0\}) =0$ and $\int_{\rd} (|x|^{2} \wedge 1) |\nu|(dx) < \infty$.
\end{definition}

Note that the triplet $(A,\nu ,\gm)$ in this case is also unique if each component exists and that infinitely divisible distributions on $\rd$ are 
quasi-infinitely divisible if and only if the negative part of $\nu$ in the Jordan decomposition is zero.
The measure $\nu$ is called {\it quasi}-L\'evy measure.
Without the name of quasi-infinitely divisible distributions, its property was used in books and papers such as Gnedenko and Kolmogorov \cite{GK68} (p.81),
Linnik and Ostrovskii \cite{LO} (Chapter 6, Section 7) and Niedbalska-Rajba \cite{N-R}.
(See, \cite[Introduction]{LSm}.)
Some of main results in \cite{LSm} were to classify distributions appeared in their story into the classes $ID(\R)$, $ID^0(\R)$ and $ID^{00}(\R)$ which are 
the class of infinitely divisible distributions on $\R$, the class of quasi-infinitely divisible but non-infinitely divisible distributions on $\R$ and the class
of distributions on $\R$ which are not quasi-infinitely divisible, respectively.

\vskip3mm

On the other hand, let $\zeta(s)$, $s\in\mathbb{C}$, be the Riemann zeta function, then it is known that a normalized function 
$f(t):=\zeta(\sigma +{\rm {i}}t)/\zeta(\sigma)$, where $s:=\sigma +{\rm {i}}t$, $\sigma >1$ and $t\in\R$, 
is an $\R$-valued infinitely divisible characteristic function.
(See, e.g.\,\cite{GK68}.)
As to generalize it to multidimensional case, we have introduced some new multivariable zeta functions and studied 
when their normalized functions can be characteristic functions in \cite{AN11k}, and afterwards, in \cite{AN12e} and \cite{AN12s}.
As a generalization of the series representations of zeta functions, we have introduced the following.

\begin{definition}[Multidimensional Shintani zeta function, \cite{AN12s}]\label{Def}
Let $d,m,r\in\N$, $\vs\in\mathbb{C}^d$ and $(n_1, \ldots , n_r)\in\mathbb{Z}_{\ge 0}^{r}$.
For $\ld_{lj}, u_j > 0$, $\vc_l \in {\mathbb{R}}^d$, where $1\le j\le r$ and $1\le l\le m$, and 
a function $\tht (n_1, \ldots , n_r)\in{\mathbb{C}}$
satisfying $|\tht (n_1, \ldots , n_r)| = O((n_1+ \cdots +n_r)^{\ep})$, for any $\ep >0$, 
we define a multidimensional Shintani zeta function given by
\begin{equation}
Z_S (\vs) := \sum_{n_1 ,\ldots, n_r =0}^{\infty} 
\frac{\tht (n_1,\ldots , n_r)}{\prod_{l=1}^m (\lambda_{l1}(n_1+u_1) 
+\cdots+\lambda_{lr}(n_r+u_r) )^{\langle \vc_l,\vs \rangle}}.
\label{eq:def2}
\end{equation}
\end{definition}

In addition, the corresponding probability distributions on $\rd$ is also defined.
Let $\tht (n_1, \ldots , n_r)$ be a nonnegative or nonpositive definite function and 
write $\vc_l=(c_{l1},\dots,c_{ld}) $ $\in {\mathbb{R}}^d$ in Definition \ref{Def}. 

\begin{definition}[Multidimensional Shintani zeta distribution, \cite{AN12s}]\label{DefSD}
For $(n_1,\dots , n_r)\in\mathbb{Z}_{\ge 0}^{r}$ and $\vsig$ satisfying 
$\min_{1\le l\le m}\langle \vc_l, \vsig\rangle >r/m$, 
we define a multidimensional Shintani zeta random variable $X_{\vsig}$ 
with probability distribution on $\rd$ given by
\begin{align*}
{\rm Pr} \Biggl(X_{\vsig}= \biggl(\,& -\sum_{l=1}^m c_{l1} \log \bigl(\lambda_{l1}(n_1+u_1) +\cdots+\lambda_{lr}(n_r+u_r)\bigr),\\
&\dots , -\sum_{l=1}^m c_{ld} \log \bigl(\lambda_{l1}(n_1+u_1) +\cdots+\lambda_{lr}(n_r+u_r)\bigr)\,\biggr) \Biggr)\\
=\ \ \ \ &\frac{\tht (n_1,\ldots , n_r)}{Z_S(\vsig)}
\prod_{l=1}^m \bigl(\lambda_{l1}(n_1+u_1) +\cdots+\lambda_{lr}(n_r+u_r) \bigr)^{-\langle \vc_l,\vsig \rangle}.
\end{align*}
\end{definition}

We may not have the infinite divisibility of distributions by series representation like above, 
so that we treated Euler products as to obtain infinitely divisible zeta distributions on $\rd$ in \cite{AN12e}.
The definition of the product is as follows.

\begin{definition}[Multidimensional polynomial Euler product, \cite{AN12e}]\label{def:EP}
Let $d,k\in\N$ and $\vs\in\mathbb{C}^d$.
For $-1 \le \alpha_h(p) \le 1$ and $\va_h \in {\mathbb{R}}^d$, $1\le h\le k$ and $p$ is a prime number,
we define multidimensional polynomial Euler product given by
\begin{equation}
Z_E (\vs) = \prod_p \prod_{h=1}^k \left( 1 - \alpha_h(p) p^{-\langle \va_h,\vs\rangle} \right)^{-1}.
\label{eq:def1}
\end{equation}
\end{definition}

In \cite{AN12e}, we have given some necessary and sufficient conditions for several multidimensional polynomial Euler products 
to generate $\rd$-valued characteristic functions, which included infinite divisibility as well.
Still, there remain many unknown properties of behaviors of multivariable zeta functions in view of probability theory such as the cases except 
under the conditions given in our papers \cite{AN12e} and \cite{AN12s}.
In this paper, we treat more simple multivariable zeta functions, 
here simple means that $\alpha_h (p) = 0$ except for some finite subset of prime numbers $p$ in \eqref{eq:def1},
and look more carefully at what happens in the same view as above. 
The functions we use are the following.

\begin{definition}
Let $p$ be a prime number. 
For $s_1, s_2\in\mathbb{C}$ with $\Re{(s_1)}, \Re{(s_2)}>0$, define functions $\zeta_p^\#$ and $\zeta_p^*$ given by
\begin{equation*}
\zeta_p^\# (\vs) = \frac{1}{(1-p^{-s_1})(1-p^{-s_2})}, \q \zeta_p^* (\vs) = \frac{1}{(1+p^{-s_1-s_2})}, 
\end{equation*}
where $\vs:=(s_1, s_2)\in\mathbb{C}^2$, respectively.
\end{definition}
Note that these functions are two-variable cases of finite Euler products which also belongs to the class of multidimensional polynomial Euler products as mentioned above.

\vskip3mm

Denote by $\wh{ID}(\R^2)$, $\wh{ID^0}(\R^2)$ and $\wh{ND}(\R^2)$ the class of $\R^2$-valued infinitely divisible characteristic functions, 
the class of $\R^2$-valued quasi-infinitely divisible but non-infinitely divisible characteristic functions and the class of $\R^2$-valued functions not 
even characteristic functions, respectively.
Our purpose is to classify some normalized functions of two-variable finite Euler products above and their products 
into $\wh{ID}(\R^2)$, $\wh{ID^0}(\R^2)$ and $\wh{ND}(\R^2)$.
The case $ID^{00}(\R^2)$ does not appear in our story.

\section{Main results}

In this section, we give our main results.
We only treat the case $\R^2$, so that we rewrite $\wh{ID}(\R^2)$, $\wh{ID^0}(\R^2)$ and $\wh{ND}(\R^2)$ by $\wh{ID}$, $\wh{ID^0}$ and $\wh{ND}$, respectively.
Proofs of theorems in this section will be given in Section 3.
As to give them, we setup some functions and variables.

\vskip3mm

Put $s_1:=\sigma_1+{\rm{i}}t_1$, $s_2:=\sigma_2+{\rm{i}}t_2$, $\vsig :=(\sigma_1,\sigma_2)$ and $\vt:=(t_1,t_2)$, 
where $\sigma_1, \sigma_2>0$ and  $t_1, t_2\in\R$.
Then, for $\vs =(s_1,s_2)$, we have $\vs=\vsig +{\rm {i}}\vt$.
Now define two functions
\begin{equation*}
\begin{split}
g_p^\# (\vsig,\vt) &:= \frac{1}{(1-p^{-(\sigma_1+{\rm{i}}t_1}))(1-p^{-(\sigma_2+{\rm{i}}t_2)})}\left(=\zeta_p^\#(\vs)\right), \\ 
g_p^* (\vsig,\vt) &:= \frac{1}{(1+p^{-(\sigma_1+{\rm{i}}t_1)-(\sigma_2+{\rm{i}}t_2)})}\left(=\zeta_p^*(\vs)\right)
\end{split}
\end{equation*}
and corresponding normalized functions
\begin{equation*}
G_p^\# (\vsig,\vt) := \frac{g_p^\# (\vsig,\vt)}{g_p^\# (\vsig,{\Vec{0}})}, \q
G_p^* (\vsig,\vt) := \frac{g_p^* (\vsig,\vt)}{g_p^* (\vsig,{\Vec{0}})} .
\end{equation*}

By following the history of the Riemann zeta function and its definable characteristic function, it seems to be natural to treat $G_p^\#$ and $G_p^*$ 
to find out whether $\zeta_p^\#$ and $\zeta_p^*$ can generate characteristic functions of distributions on $\R^2$ or not.

\begin{theorem}\label{th:0}
We have the following.
\begin{description}
 \item[(1)] $G_p^\# \in \wh{ID}$. 
 \item[(2)] $G_p^* \in \wh{ND}$. 
 \item[(3)] $G_p^\#G_p^* \in \wh{ID^0}$. 
\end{description}
\end{theorem}

Now we add following functions.
For complex variables $s_1,s_2$ satisfying the conditions above, put
\begin{equation*}
\begin{split}
&f_p (\vsig,\vt) := \frac{1}{(1-p^{-s_1-s_2})},\\
&g_p (\vsig,\vt) := \frac{1}{(1-p^{-s_1})(1-p^{-s_2})(1+p^{-s_1-s_2})},\\
&h_p (\vsig,\vt) := \frac{1}{(1+p^{-s_1})(1+p^{-s_2})(1-p^{-s_1-s_2})}.
\end{split}
\end{equation*}

and, respectively, corresponding normalized functions

\begin{equation*}
F_p(\vsig,\vt) := \frac{f_p (\vsig,\vt)}{f_p (\vsig,{\Vec{0}})}, \q G_p(\vsig,\vt) := \frac{g_p (\vsig,\vt)}{g_p (\vsig,{\Vec{0}})}, \q H_p(\vsig,\vt) := \frac{h_p (\vsig,\vt)}{h_p (\vsig,{\Vec{0}})}.
\end{equation*}

Then, we have following three theorems.

\begin{theorem}\label{th:1}
We have the following.
\begin{description}
 \item[(1)] $F_p \in \wh{ID}$. 
 \item[(2)] $G_p \in \wh{ID^0}$. 
 \item[(3)] $H_p \in \wh{ND}$. 
\end{description}
\end{theorem}

\begin{theorem}\label{th:2}
We have the following.
\begin{description}
 \item[(1)] $F_p G_p \in \wh{ID}$. 
 \item[(2)] $G_p H_p \in \wh{ID}$. 
 \item[(3)] $H_p F_p \in \wh{ND}$. 
\end{description}
\end{theorem}

\begin{theorem}\label{th:3}
Let $p$ and $q$ be distinct prime numbers. 
We have the following.
\begin{description}
 \item[(1)] $F_p G_q \in \wh{ID^0}$. 
 \item[(2)] $G_p H_q \in \wh{ND}$. 
 \item[(3)] $H_p F_q \in \wh{ND}$. 
\end{description}
\end{theorem}

\section{Proofs of Theorems}

In this section, we give the proofs of four theorems mentioned in Section 2.
In some of the proofs, we use the following known fact.

\begin{proposition}[See, e.g.\,\cite{S99}]\label{pro:Sato}
Let $\mu$ be a probability measure on $\rd$. 
Then, it holds that $|\wh\mu (t)|\le 1$ for any $t\in\rd$.\\
\end{proposition}

\subsection{Proof of Theorem \ref{th:0}}

The proof of Theorem \ref{th:0} is as follows.

\begin{proof}[Proof of Theorem \ref{th:0} $(1)$]
Let $l\in\{1,2\}$. 
Note that $|-p^{-\sigma_l +{\rm {i}}t_l}|=|-p^{-\sigma_l}|<1$ for $\sigma_l >0$.
By the Taylor series of $\log \left(1+(-x)\right)=\sum_{r=1}^{\infty}(-1)^{r+1}r^{-1}(-x)^r=-\sum_{r=1}^{\infty}r^{-1}x^r$, $|x|<1$,  
we have
\begin{equation*}
\begin{split}
\log G_p^\# (\vsig,\vt) &= - \log \frac{(1-p^{-\sigma_1-{\rm{i}}t_1})(1-p^{-\sigma_2-{\rm{i}}t_2})}{(1-p^{-\sigma_1})(1-p^{-\sigma_2})} \\ 
&= \, \sum_{r=1}^{\infty} \frac{1}{r} p^{-r\sigma_1} \bigl( p^{-r{\rm{i}}t_1} -1 \bigr) +
\sum_{r=1}^{\infty} \frac{1}{r} p^{-r\sigma_2} \bigl( p^{-r{\rm{i}}t_2} -1 \bigr)\\
&= \, \int_{\R^2} \big( e^{-{\rm{i}}\langle\vt , x\rangle} -1 \bigr) N^{G_p^\#}_{\vsig}(dx),
\end{split}
\end{equation*}
where 
$$
N^{G_p^\#}_{\vsig}(dx):=\sum_{r=1}^{\infty} 
\frac{1}{r} p^{-r\sigma_1} \delta_{\log p^r (1,0)} (dx)+ \sum_{r=1}^{\infty} \frac{1}{r} p^{-r\sigma_2} \delta_{\log p^r (0,1)} (dx)
$$
and $\delta_{\al\va}$ is the delta measure at $\al\va$ for $\al\in\R$ and $\va\in\R^2$.
Here $N^{G_p^\#}_{\vsig}$ is a finite L\'evy measure on $\R^2$ since 
\begin{align*}
\int_{\R^2}N^{G_p^\#}_{\vsig}(dx)&=\sum_{l=1,2}\sum_{r=1}^{\infty} \frac{1}{r} p^{-r\sigma_l}
\le \sum_{l=1,2}\left(p^{-\sigma_l}+\sum_{r=2}^{\infty}\frac{1}{r}p^{-r\sigma_l} \right)\\
&\le \sum_{l=1,2}\left(p^{-\sigma_l}+\int_1^{\infty}\frac{1}{x}p^{-x\sigma_l} dx\right) \le \sum_{l=1,2} \left(p^{-\sigma_l}+\left(\sigma_l\log p\right)^{-1}\right)<\infty.
\end{align*}
Hence $G_p^\#$ is an infinitely divisible (actually, it is compound Poisson) characteristic function.
\end{proof}

\begin{proof}[Proof of Theorem \ref{th:0} $(2)$]
Obviously, there exists $\vt_0 = (t_0,t_0)$ such that $p^{-2{\rm{i}}t_0} =-1$. 
Then, we have
$$
|G_p^*(\vsig,\vt_0)| = \left| \frac{1+p^{-\sigma_1-\sigma_2}}{1+p^{-\sigma_1-\sigma_2-2{\rm{i}}t_0}} \right|
=\frac{1+p^{-\sigma_1-\sigma_2}}{1-p^{-\sigma_1-\sigma_2}} >1.
$$
By Proposition \ref{pro:Sato}, we have $G_p^* \in \wh{ND}$.
\end{proof}

\begin{proof}[Proof of Theorem \ref{th:0} (3)]
First we show that $G_p^\#G_p^*$ is a characteristic function.
We have
$$
g_p^\#(\vsig,\vt)g_p^*(\vsig,\vt) =\frac{1}{(1-p^{-2s_1-2s_2})} \frac{1-p^{-s_1-s_2}}{(1-p^{-s_1})(1-p^{-s_2})}.
$$
For any $X,Y$ with $|X|,|Y|<1$, we also have
$$
\frac{1-XY}{(1-X)(1-Y)} = \frac{1}{1-X} + \frac{1}{1-Y} - 1 = 1 + \sum_{n=1}^\infty \bigl( X^n+Y^n \bigr).
$$
Therefore, we obtain
$$
g_p^\#(\vsig,\vt)g_p^*(\vsig,\vt) = \sum_{l=0}^{\infty} \frac{1}{p^{2l(s_1+s_2)}}
\left( 1 + \sum_{m=1}^{\infty} \frac{1}{p^{ms_1}} +\sum_{n=1}^{\infty} \frac{1}{p^{ns_2}} \right) =
\sum_{l,m,n=1}^{\infty} \frac{A(l)A(m,n)}{l^{s_1+s_2}m^{s_1}n^{s_2}},
$$
where, for $a,b,c \in \N$,
$$
A(l) := 
\begin{cases}
1 & l = 1, \, p^{2a}, \\
0 & \mbox{otherwise},
\end{cases} \qq
A(m,n) := 
\begin{cases}
1 & (m,n)=(1,1), \, (1,p^b), \, (p^c,1), \\
0 & \mbox{otherwise}.
\end{cases}
$$
By Definitions \ref{Def} and \ref{DefSD}, $G_p^\#G_p^*$ is a characteristic function which belongs to multidimensional Shintani zeta distribution.

Next we show that $G_p^\#G_p^*$ is quasi-infinitely divisible.
We have
\begin{equation*}
\begin{split}
&\log G_p^\#(\vsig,\vt)G_p^* (\vsig,\vt) 
= - \log \frac{(1-p^{-\sigma_1-{\rm{i}}t_1})(1-p^{-\sigma_2-{\rm{i}}t_2})(1+p^{-\sigma_1-{\rm{i}}t_1-\sigma_2-{\rm{i}}t_2})}
{(1-p^{-\sigma_1})(1-p^{-\sigma_2})(1+p^{-\sigma_1-\sigma_2})} \\ 
&= \,\sum_{r=1}^{\infty} \frac{1}{r} p^{-r\sigma_1} \bigl( p^{-r{\rm{i}}t_1} -1 \bigr) +
\sum_{r=1}^{\infty} \frac{1}{r} p^{-r\sigma_2} \bigl( p^{-r{\rm{i}}t_2} -1 \bigr)
+\sum_{r=1}^{\infty} \frac{(-1)^r}{r} p^{-r(\sigma_1+\sigma_2)} \bigl( p^{-r{\rm{i}}(t_1+t_2)} -1 \bigr)\\
&= \, \int_{\R^2} \big( e^{-{\rm{i}}\langle\vt , x\rangle} -1 \bigr) N^{G_p^\#G_p^*}_{\vsig}(dx),
\end{split}
\end{equation*}
where
\begin{align*}
N^{G_p^\#G_p^*}_{\vsig}(dx):=&\sum_{r=1}^{\infty} \frac{1}{r} p^{-r\sigma_1} \delta_{\log p^r (1,0)} (dx)
 +\sum_{r=1}^{\infty} \frac{1}{r} p^{-r\sigma_2} \delta_{\log p^r (0,1)} (dx)\\
& +\sum_{r=1}^{\infty} \frac{(-1)^r}{r} p^{-r(\sigma_1+\sigma_2)} \delta_{\log p^r (1,1)} (dx).
\end{align*}
The measure $N^{G_p^\#G_p^*}_{\vsig}$ is a quasi-L\'evy measure on $\R^2$ since the third term is a signed measure which can not be canceled by the other terms and, as in the proof of Theorem \ref{th:0} $(1)$, we easily obtain $\int_{\R^2}\bigl|N^{G_p^\#G_p^*}_{\vsig}\bigr|(dx)<\infty$.
Hence $G_p^\#G_p^*\in\wh{ID^0}$. 
\end{proof}

\subsection{Proof of Theorem \ref{th:1}}
The statement (2) of Theorem \ref{th:1} coincides with (3) of Theorem \ref{th:0}. 
Thus we only have to show (1) and (3). 
\begin{proof}[Proof of Theorem \ref{th:1} (1)]
We have
\begin{equation*}
\begin{split}
\log F_p (\vsig,\vt) = \int_{\R^2} \big( e^{-{\rm{i}}\langle\vt , x\rangle} -1 \bigr) 
\sum_{r=1}^{\infty} \frac{1}{r} p^{-r(\sigma_1+\sigma_2)} \delta_{\log p^r (1,1)} (dx) .
\end{split}
\end{equation*}
By following the proof of Theorem \ref{th:0} $(1)$, we have $F_p \in \wh{ID}$. 
\end{proof}

\begin{proof}[Proof of Theorem \ref{th:1} (3)]
As in the proof of Theorem \ref{th:0} $(2)$, for $\vt_0 = (t_0,t_0)$ such that $p^{-{\rm{i}}t_0} =-1$, we have
$$
|H_p(\vsig,\vt_0)| = 
\frac{(1+p^{-\sigma_1})(1+p^{-\sigma_2})(1-p^{-\sigma_1-\sigma_2})}{(1-p^{-\sigma_1})(1-p^{-\sigma_2})(1-p^{-\sigma_1-\sigma_2})} = \frac{(1+p^{-\sigma_1})(1+p^{-\sigma_2})}{(1-p^{-\sigma_1})(1-p^{-\sigma_2})} >1.
$$
By following the proof of Theorem \ref{th:0} $(2)$, we have $H_p \in \wh{ND}$.

\end{proof}

\subsection{Proof of Theorem \ref{th:2}}
Now we show Theorem \ref{th:2}.
\begin{proof}[Proof of Theorem \ref{th:2} (1) and (2)]
We have
\begin{align*}
\log F_p (\vsig,\vt)G_p (\vsig,\vt) &= \int_{\R^2} \big( e^{-{\rm{i}}\langle\vt , x\rangle} -1 \bigr) N^{F_pG_p}_{\vsig}(dx) ,\\
\log G_p(\vsig,\vt_0)H_p(\vsig,\vt_0) &= \int_{\R^2} \big( e^{-{\rm{i}}\langle\vt , x\rangle} -1 \bigr) N^{G_pH_p}_{\vsig}(dx) ,
\end{align*}
where 
\begin{align*}
N^{F_pG_p}_{\vsig}(dx) := &\sum_{r=1}^{\infty}\frac{1}{r} p^{-r\sigma_1} \delta_{\log p^r (1,0)} (dx) 
+\sum_{r=1}^{\infty} \frac{1}{r} p^{-r\sigma_2} \delta_{\log p^r (0,1)} (dx)\\
& + \sum_{r=1}^{\infty} \frac{1}{r} p^{-2r(\sigma_1+\sigma_2)} \delta_{\log p^{2r} (1,1)} (dx), \\
N^{G_pH_p}_{\vsig}(dx) :=&\sum_{r=1}^{\infty}\frac{1}{r} p^{-2r\sigma_1} \delta_{\log p^{2r} (1,0)} (dx) + \sum_{r=1}^{\infty} \frac{1}{r} p^{-2r\sigma_2} \delta_{\log p^{2r} (0,1)} (dx)\\
&+ \sum_{r=1}^{\infty} \frac{1}{r} p^{-2r(\sigma_1+\sigma_2)} \delta_{\log p^{2r} (1,1)} (dx)
\end{align*}
are finite L\'evy measures on $\R^2$.
Again, by following the proof of Theorem \ref{th:0} $(1)$, we have $F_pG_p, G_pH_p\in\wh{ID}$.
\end{proof}

\begin{proof}[Proof of Theorem \ref{th:2} (3)]
Again, for $\vt_0$ such that $p^{-{\rm{i}}t_0} =-1$, we have 
$$
|H_p(\vsig,\vt_0)F_p(\vsig,\vt_0)| = 
\frac{(1+p^{-\sigma_1})(1+p^{-\sigma_2})(1-p^{-2\sigma_1-2\sigma_2})^2}{(1-p^{-\sigma_1})(1-p^{-\sigma_2})(1-p^{-2\sigma_1-2\sigma_2})^2} = \frac{(1+p^{-\sigma_1})(1+p^{-\sigma_2})}{(1-p^{-\sigma_1})(1-p^{-\sigma_2})} >1.
$$
By following the proof of Theorem \ref{th:0} $(2)$, we have $H_pF_p\in \wh{ND}$.
\end{proof}

\subsection{Proof of Theorem \ref{th:3}}
Finally, we show Theorem \ref{th:3}.
\begin{proof}[Proof of Theorem \ref{th:3} (1)]
By Theorem \ref{th:1} $(1)$ and $(2)$, $F_pG_q$ is a characteristic function.
So that we only have to show the quasi-infinite divisibility of $F_pG_q$.
We have
\begin{equation*}
\log F_p (\vsig,\vt) G_q (\vsig,\vt) = \int_{\R^2} \big( e^{-{\rm{i}}\langle\vt , x\rangle} -1 \bigr) N^{F_pG_q}_{\vsig}(dx),
\end{equation*}
where
\begin{equation*}
\begin{split}
N^{F_pG_q}_{\vsig}(dx):=& \sum_{r=1}^{\infty} \frac{1}{r} p^{-r(\sigma_1+\sigma_2)} \delta_{\log p^r (1,1)} (dx)
+\sum_{r=1}^{\infty} \frac{1}{r} q^{-r\sigma_1} \delta_{\log q^r (1,0)} (dx)\\
& + \sum_{r=1}^{\infty} \frac{1}{r} q^{-r\sigma_2} \delta_{\log q^r (0,1)} (dx)
 + \sum_{r=1}^{\infty} \frac{(-1)^r}{r} q^{-r(\sigma_1+\sigma_2)} \delta_{\log q^r (1,1)} (dx).
\end{split}
\end{equation*}
Since $r_1 \log p_1 = r_2 \log p_2$ if and only if $r_1=r_2$ and $p_1=p_2$, the fourth term of $N^{F_pG_q}_{\vsig}$ is a
signed measure which can not be canceled by the other terms.
We also have that $\int_{\R^2}\bigl|N^{F_pG_q}_{\vsig}\bigr|(dx)<\infty$ as in the proof of Theorem \ref{th:0} $(1)$.
Thus the measure $N^{F_pG_q}_{\vsig}$ is a quasi-L\'evy measure on $\R^2$.
As similar as the proof of Theorem \ref{th:1} $(2)$, we obtain $F_pG_q\in\wh{ID^0}$.
\end{proof}

For the proof of Theorem \ref{th:3} $(2)$ and $(3)$, we use linear independence of real numbers and the Kronecker's approximation theorem.

\vskip3mm

It is called that real numbers $\theta_1,\ldots,\theta_n$ are {\it linearly independent over the rationals} if $\sum_{k=1}^n c_k \theta_k=0$ with rational multipliers $c_1,\ldots,c_n$ implies $c_1 = \cdots = c_n =0$. 
Put
\begin{equation}
\theta_k := \frac{\log p_k}{2\pi}, \qq 1 \le k \le n,
\label{eq:deftheta}
\end{equation}
where $p_1, \ldots, p_n$ are the first $n$ primes. 
Then $\theta_k$ are linearly independent over the rationals and it can be shown by this way. 
When we suppose $\sum_{k=1}^n c_k \theta_k =0$, namely, $\log (p_1^{c_1} \cdots p_n^{c_n}) =0$, it implies that $p_1^{c_1} \cdots p_n^{c_n}=1$. 
Hence we obtain $c_1 = \cdots = c_n =0$ by the fundamental theorem of arithmetic (or the unique-prime-factorization theorem). 

The following proposition is called the (first form of) Kronecker's approximation theorem. 
\begin{proposition}[{\cite[Theorem 7.9]{Apo2}}]\label{pro:kroap1}
If $\phi_1 , \ldots ,\phi_n$ are arbitrary real numbers, if real numbers $\theta_1,\ldots,\theta_n$ are linearly independent over the rationals, and if $\varepsilon>0$ is arbitrary, then there exists a real number $t$ and rationals $h_1, \ldots, h_n$ such that
$$
|t\theta_k - h_k - \phi_k| < \varepsilon, \qq 1 \le k \le n. 
$$
\end{proposition}

Now we prove Theorem \ref{th:3} $(2)$ and $(3)$.

\begin{proof}[Proof of Theorem \ref{th:3} (2)]
We have
\begin{equation*}
\begin{split}
2\log &\bigl|G_p(\vsig,\vt) H_q(\vsig,\vt)\bigr| = 
\log \bigl(G_p(\vsig,\vt) G_p(\vsig,-\vt) H_q(\vsig,\vt) H_q(\vsig,-\vt) \bigr) \\ 
= \,& \sum_{r=1}^{\infty} \frac{1}{r} p^{-r\sigma_1} \bigl( p^{r{\rm{i}}t_1} + p^{-r{\rm{i}}t_1} -2 \bigr) +
\sum_{r=1}^{\infty} \frac{1}{r} p^{-r\sigma_2} \bigl( p^{r{\rm{i}}t_2} + p^{-r{\rm{i}}t_2} -2 \bigr) \\ 
\,&+\sum_{r=1}^{\infty} \frac{(-1)^r}{r} p^{-r(\sigma_1+\sigma_2)} 
\bigl( p^{r{\rm{i}}(t_1+t_2)} + p^{-r{\rm{i}}(t_1+t_2)} -2 \bigr) \\ 
\,&+\sum_{r=1}^{\infty} \frac{(-1)^r}{r} q^{-r\sigma_1} \bigl( q^{r{\rm{i}}t_1} + q^{-r{\rm{i}}t_1} -2 \bigr) +
\sum_{r=1}^{\infty} \frac{(-1)^r}{r} q^{-r\sigma_2} \bigl( q^{r{\rm{i}}t_2} + q^{-r{\rm{i}}t_2} -2 \bigr) \\ 
\,&+\sum_{r=1}^{\infty} \frac{1}{r} q^{-r(\sigma_1+\sigma_2)} 
\bigl( q^{r{\rm{i}}(t_1+t_2)} + q^{-r{\rm{i}}(t_1+t_2)} -2 \bigr).
\end{split}
\end{equation*}
Let $\sum_{r>R}$ be a sum taken over $r>R$.
For any $\ep >0$, we can see that there exists an integer $R$ such that $\sum_{r>2R}r^{-1}p^{-r\sigma_j}< \varepsilon$ and $\sum_{r>2R}r^{-1}q^{-r\sigma_j}< \varepsilon$, where $\sigma_j = \sigma_1, \sigma_2$ or $\sigma_1+\sigma_2$. 

In the view of $p^{{\rm{i}}t}= e^{{\rm{i}}t\log p} = e^{2\pi {\rm{i}}t\theta_1}$, $q^{{\rm{i}}t} = e^{2\pi {\rm{i}}t\theta_2}$, \eqref{eq:deftheta} and by Proposition \ref{pro:kroap1}, for any $\ep '>0$ independent of $\ep$ and $R$, there exists $t_0 \in \R$ such that
$$
|p^{{\rm{i}}t_0}-1| < \varepsilon',  \q \mbox{and} \q 
|q^{{\rm{i}}t_0}+1| < \varepsilon'.
$$
By the factorization $x^r-1 = (x-1)(x^{r-1}+\cdots+1)$, for any $1 \le r \le 2R$, we have $|p^{r{\rm{i}}t_0}-1| < r\varepsilon' \le 2R\varepsilon'$. 
Similarly, by the factorization $x^{2k-1}+1 = (x+1)(x^{2k-2}-x^{2k-3}+\cdots+1)$ when $r=2k-1\in 2\N -1$ with $k\le R$, one has $|q^{(2k-1){\rm{i}}t_0}+1| < (2k-1)\varepsilon' \le 2R\varepsilon'$. 
Also, by the factorization $x^{2k}-1 = (x+1)(x-1)(x^{2k-2}+x^{2k-4}+\cdots+1)$ when $r=2k\in 2\N$ with $k\le R$, it holds that $|q^{2k{\rm{i}}t_0}-1| < 2k \varepsilon' \le 2R\varepsilon'$.
Hence there exits $t_0\in\R$ such that
\begin{equation*}
\begin{split}
-4R\varepsilon' < p^{r{\rm{i}}t_0} + p^{-r{\rm{i}}t_0} -2 \le 0, \q
&1 \le r \le 2R ,\\
-8R\varepsilon' < p^{2r{\rm{i}}t_0} + p^{-2r{\rm{i}}t_0} -2 \le 0, \q
&1 \le r \le 2R ,\\
-4 -4R\varepsilon' < q^{(2k-1){\rm{i}}t_0} + q^{-(2k-1){\rm{i}}t_0} - 2 < -4 + 4R\varepsilon' , \q
&1 \le k \le R,\\ 
-4R\varepsilon' \le q^{2k{\rm{i}}t_0} + q^{-2k{\rm{i}}t_0} - 2 \le 0, \q
&1 \le k \le R,\\
-8R\varepsilon' \le q^{2k{\rm{i}}t_0} + q^{-2k{\rm{i}}t_0} - 2 \le 0, \q
&1 \le k \le 2R.
\end{split}
\end{equation*}
Hence, for $\vt_0:=(t_0,t_0)\in\R^2$, we have 
\begin{equation}
\begin{split}
\log \bigl(G_p(\vsig,\vt_0)& G_p(\vsig,-\vt_0) \bigr)\\  
&>- 12 \varepsilon-4R\varepsilon' \sum_{r=1}^{2R}  \frac{1}{r} \bigl( p^{-r\sigma_1}
 + p^{-r\sigma_2} \bigr) -8R\varepsilon' \sum_{r=1}^{2R}  \frac{1}{r} p^{-r\sigma_1-r\sigma_2} ,
\end{split}
\label{eq:Gp}
\end{equation}
\begin{equation}
\begin{split}
&\log \bigl(H_q(\vsig,\vt_0) H_q(\vsig,-\vt_0) \bigr) > - 12 \varepsilon
-8R\varepsilon' \sum_{r=1}^{2R}  \frac{1}{r} p^{-r\sigma_1-r\sigma_2}  \\
&+(4 - 4R\varepsilon') \sum_{k=1}^{R}  \frac{1}{2k-1} \bigl( q^{-(2k-1)\sigma_1} + q^{-(2k-1)\sigma_2} \bigr) 
-4R\varepsilon' \sum_{k=1}^{R}  \frac{1}{2k} \bigl( q^{-2k\sigma_1} + q^{-2k\sigma_2} \bigr).
\end{split}
\label{eq:Hq}
\end{equation}
Therefore we obtain
\begin{align}\label{EQ}
2\log |G_p(\vsig,\vt_0) H_q(\vsig,\vt_0)|>-12\ep-4RC'\ep'+C,
\end{align}
where
\begin{align*}
C':=&\sum_{r=1}^{2R}  \frac{1}{r} \bigl( p^{-r\sigma_1}+ p^{-r\sigma_2} \bigr) 
+2\sum_{r=1}^{2R}  \frac{1}{r} p^{-r\sigma_1-r\sigma_2}+2 \sum_{r=1}^{2R}  \frac{1}{r} p^{-r\sigma_1-r\sigma_2}\\
&+\sum_{k=1}^{R}  \frac{1}{2k-1} \bigl( q^{-(2k-1)\sigma_1} + q^{-(2k-1)\sigma_2} \bigr)+ \sum_{k=1}^{R}  \frac{1}{2k} \bigl( q^{-2k\sigma_1} + q^{-2k\sigma_2} \bigr)>0,\\
C:=&\,4\sum_{k=1}^{R}  \frac{1}{2k-1} \bigl( q^{-(2k-1)\sigma_1} + q^{-(2k-1)\sigma_2} \bigr) >0.
\end{align*}
Now suppose $\varepsilon$ is sufficiently small and $\varepsilon'$ such that $4R\varepsilon' < \varepsilon$. 
Then, we finally obtain $\log |G_p(\vsig,\vt_0) H_q(\vsig,\vt_0)|>0$ by the third term of the right-hand side of \eqref{EQ}. 
Hence we have $G_pH_q\in\wh{ND}$ by Proposition \ref{pro:Sato}.
\end{proof}

\begin{proof}[Proof of Theorem \ref{th:3} (3)]
We have
\begin{equation*}
\begin{split}
2\log &\left|F_p(\vsig,\vt) H_q(\vsig,\vt)\right| = 
\log \bigl(F_p(\vsig,\vt) F_p(\vsig,-\vt) H_q(\vsig,\vt) H_q(\vsig,-\vt) \bigr) \\ 
= \,&
\sum_{r=1}^{\infty} \frac{1}{r} p^{-r(\sigma_1+\sigma_2)} 
\bigl( p^{r{\rm{i}}(t_1+t_2)} + p^{-r{\rm{i}}(t_1+t_2)} -2 \bigr) \\ &+
\sum_{r=1}^{\infty} \frac{(-1)^r}{r} q^{-r\sigma_1} \bigl( q^{r{\rm{i}}t_1} + q^{-r{\rm{i}}t_1} -2 \bigr) +
\sum_{r=1}^{\infty} \frac{(-1)^r}{r} q^{-r\sigma_2} \bigl( q^{r{\rm{i}}t_2} + q^{-r{\rm{i}}t_2} -2 \bigr) \\ 
&+
\sum_{r=1}^{\infty} \frac{1}{r} q^{-r(\sigma_1+\sigma_2)} 
\bigl( q^{r{\rm{i}}(t_1+t_2)} + q^{-r{\rm{i}}(t_1+t_2)} -2 \bigr) .
\end{split}
\end{equation*}
Now we can follow the proof of Theorem \ref{th:3} $(2)$, there exists $\vt_0\in\R^2$ such that \eqref{eq:Hq} and
\begin{equation}\label{eq:Fp}
\log \bigl(F_p(\vsig,\vt_0) F_p(\vsig,-\vt_0) \bigr) > - 4 \varepsilon
-8R\varepsilon' \sum_{r=1}^{2R}  \frac{1}{r} p^{-r\sigma_1-r\sigma_2}, 
\end{equation}
instead of \eqref{eq:Gp}, hold.
This implies there exists $\vt_0\in\R^2$ such that $\log \bigl|F_p(\vsig,\vt_0)$ $H_q(\vsig,\vt_0)\bigr|>0$.
Hence we also obtain $F_pH_q\in\wh{ND}$ in the same way as the proof of Theorem \ref{th:3} $(2)$.
\end{proof}



 

\begin{thebibliography}{1}

\bibitem{AN12e}
{\rm T.~Aoyama \and T.~Nakamura}, `Multidimensional polynomial Euler products and infinitely divisible distributions on $\rd$', preprint (2012).

\bibitem{AN12s}
{\rm T.~Aoyama \and T.~Nakamura}, `Multidimensional Shintani zeta functions and zeta distributions on $\rd$', preprint (2012).

\bibitem{AN11k}
{\rm T.~Aoyama \and T.~Nakamura}, `Zeros of zeta functions and zeta distributions on $\rd$', to appear in {\em Functions in Number Theory and Their Probabilistic Aspects -- Kyoto 2010} (2011).

\bibitem{Apo2}
{\rm T.~M.~Apostol}, {\em Modular functions and Dirichlet series in Number Theory } (Graduate Texts in Mathematics 41, Springer, 1990). 

\bibitem{GK68}
{\rm B.~V.~Gnedenko \and A.~N.~Kolmogorov}, {\em Limit Distributions for Sums of Independent Random Variables (Translated from the Russian by
Kai Lai Chung)} (Addison-Wesley, 1968).

\bibitem{LSa}
{\rm A.~Lindner \and K.~Sato}, `Continuity properties and infinite divisibility of stationary distributions of some generalized Ornstein--Uhlenbeck processes',
{\em Ann. Probab.} 37 (2009) 250--274.

\bibitem{LSm}
{\rm A.~Lindner \and K.~Sato}, `Properties of stationary distributions of a sequence of generalized Ornstein-Uhlenbeck processes', {\em Math. Nachr.} 
284 (2011) 2225--2248.

\bibitem{LO}
{\rm J.~V.~Linnik \and I.~.~Ostrovskii}, `Decompositions of Random Variables and Vectors', (American Mathematical Society, Providence, RI, 1977 
(Translation from the Russian original of 1972)).

\bibitem{N-R}
{\rm T.~Niedbalska-Rajba}, `On decomposability semigroups on the real line', {\em Colloq. Math.} 44 (1981) 347--358.

\bibitem{S99} 
{\rm K.~Sato}, \textit{L\'evy Processes and Infinitely Divisible Distributions}, 
Cambridge University Press, 1999.

\end{thebibliography}
\end{document}